\documentclass[12pt]{article}

\usepackage{mathrsfs}
\usepackage{amsmath}
\usepackage{amsfonts}
\usepackage{amssymb}
\usepackage{verbatim}
\usepackage{amsmath, amsthm, amsfonts, amssymb}
\usepackage{nicefrac}

\usepackage[usenames, dvipsnames]{color}
\usepackage{latexsym}
\usepackage{MnSymbol}
\usepackage{enumerate}

\newcommand{\N}{{\mathbb N}}
\newcommand{\R}{{\mathbb R}}
\newcommand{\C}{{\mathbb C}}

\newcommand{\abs}[1]{|#1|}
\newcommand{\Abs}[1]{\big|#1\big|}
\newcommand{\norm}[1]{\|#1\|}
\newcommand{\half}{{\textstyle \frac{1}{2}}}
\newcommand{\vol}{{\operatorname{Vol}}}

\newcommand{\supp}{{\operatorname{Supp\,}}}

\newcommand{\ccal}{\mathcal{C}}
\newcommand{\dcal}{\mathcal{D}}

\newcommand{\gcal}{\mathcal{G}}
\newcommand{\hcal}{\mathcal{H}}

\newcommand{\ocal}{\mathcal{O}}

\newcommand{\scal}{\mathcal{S}}

\def \-{\bar}

\newcommand{\kahler}{K\"{a}hler}
\newcommand{\szego}{Szeg\"{o}}

\newcommand{\al}{\alpha}
\newcommand{\be}{\beta}
\newcommand{\ep}{\varepsilon}
\newcommand{\de}{\delta}
\newcommand{\De}{\Delta}
\newcommand{\ga}{\gamma}
\newcommand{\Ga}{\Gamma}
\newcommand{\la}{\lambda}
\newcommand{\La}{\Lambda}
\newcommand{\om}{\omega}
\newcommand{\Om}{\Omega}
\newcommand{\si}{\sigma}

\newcommand{\pa}{\partial}
\newcommand{\di}{\displaystyle}

\newtheorem{theorem}{Theorem}[section]
\newtheorem*{thm}{Theorem}
\newtheorem{cor}[theorem]{Corollary}
\newtheorem{lem}[theorem]{Lemma}

\newtheorem{rmk}[theorem]{Remark}
\newtheorem{prop}[theorem]{Proposition}
\newtheorem{question}[theorem]{Question}
\newtheorem{definition}[theorem]{Definition}

\newenvironment{example}{\medskip\noindent{\it Example:\/}}{\medskip}
\newenvironment{rem}{\medskip\noindent{\it Remark:\/}}{\medskip}
\newenvironment{claim}{\medskip\noindent{\it Claim:\/}}{\medskip}

\numberwithin{equation}{section}

\date{}

\begin{document}

\title{\bf Holomorphic line Bundles over a Tower of Coverings}

\author{
\ \ Yuan Yuan\footnote{ Supported in part by National Science Foundation grant DMS-1412384} \ \
and \ \ Junyan Zhu}

\vspace{3cm} \maketitle

\begin{abstract}
{\small We study a tower of normal coverings over a compact K\"ahler manifold with holomorphic line bundles. When the line bundle is sufficiently positive, we obtain an effective estimate, which implies the Bergman stability. 
As a consequence, we deduce the equidistribution for zero currents of random holomorphic sections. Furthermore, we obtain a variance estimate for those random zero currents, which yields the almost sure convergence under some geometric condition.
}
\end{abstract}

\maketitle

\section{Introduction}
Let $(M, g)$ be a Riemannian manifold with a complete Riemannian metric $g$.
 Suppose that its fundamental group $\Ga=\pi_1(M)$ admits a tower of normal subgroups: $\Ga=\Ga_0\supsetneq\Ga_1\supsetneq\cdots\supsetneq\Ga_j\supsetneq\cdots$ satisfying $2\leq[\Ga_j:\Ga_{j+1}]<\infty$ for each $j\geq0$ and $\bigcap_{j=0}^\infty\Ga_j=\{1\}$. Let $\tilde{M}$ denote the universal covering of $M$. Then $\Ga$ acts on $\tilde{M}$ as a group of deck transformations, which is free and properly discontinuous. Denote $\tilde{M}/\Ga_j$ by $M_j$ and we thus obtain a tower of normal coverings: $\tilde{M}\stackrel{p_j}\longrightarrow M_j\stackrel{q_j}\longrightarrow M_0=M$, where $p_j$ and $q_j$ denote the covering maps satisfying $q_j\circ p_j=p_0$ for all $j\geq0$. Furthermore, for each $j\geq0$, the group action $\Ga\curvearrowright\tilde{M}$ induces $\Ga/\Ga_j\curvearrowright M_j$. The differential structure and 
the Riemannian metric on each $M_j$ and $\tilde{M}$ are determined by those on $M$ via the covering maps $q_j$ and $p_0$.
It is a classical result that every Riemannian manifold whose fundamental group is isomorphic to a finitely generated subgroup of $SL(n, \mathbb{C})$ admits a tower of coverings (cf. \cite{Bo}). There have been a lot of important works studying the
asymptotic behaviors of various topological, geometrical and spectral properties for the tower of coverings of compact Riemannian manifolds (cf. \cite{CG} \cite{DW} \cite{Don1} \cite{Kaz} \cite{Ye1} and etc).

\medskip

In this paper, we are interested in 
the random complex geometry over a tower of coverings. Our motivations come from a series of works by Shiffman, Zelditch and their coauthors (cf. \cite{BSZ} \cite{SZ1} \cite{SZ2} \cite{SZ3} \cite{SZ4} and etc),  as well as the recent paper by Lu and Zelditch \cite{LZ}.
 Let $(M,\om)$ be a compact \kahler\ manifold of complex dimension $n$ with volume form $dV=\frac{\om^n}{n!}$. For simplicity, we still use $\om$ and $dV$ to denote their counterparts on each level of the tower of coverings. Since the covering indices $I_j:=[\Ga:\Ga_j]<\infty$, each $q_j:M_j\to M$ is a finite-sheeted covering. Hence $M_j$'s are all compact.

\medskip

Let $E$ be a holomorphic line bundle over $M$ with a smooth Hermitian metric $h_E$. By abuse of notation, we denote the pullback line bundles $(q_j^*E, q_j^*h_E)$ and $(p_0^*E, p_0^*h_E)$ still by $(E,h_E)$. Then we call $\{(M_j,E)\}$ a tower of normal coverings with line bundles. Let $\Pi_{j,E}$ be the Bergman kernel of the line bundle $(E,h_E)\to M_j$ and $\tilde{\Pi}_E$ be the $L^2$-Bergman kernel of $(E,h_E)\to\tilde{M}$. The base locus of $E\to M_j$ (respectively the $L^2$-base locus of $E\to\tilde{M}$) is denoted by $B_{j,E}$ (respectively $\tilde{B}_E$). The Bergman metric $\Om_{j,E}$ (respectively $\tilde{\Om}_E$) is a smooth positive $(1,1)$-form defined over $M_j\setminus B_{j,E}$ (respectively $\tilde{M}\setminus\tilde{B}_E$). As the Bergman kernel $\Pi_{j,E}$ (respectively $\tilde{\Pi}_E$) is invariant under the group action $\Ga/\Ga_j\curvearrowright M_j$ (respectively $\Ga\curvearrowright\tilde{M}$)(cf. \S 2) while $M_j/(\Ga/\Ga_j)=M$ (respectively $\tilde{M}/\Ga=M$), $\Pi_{j,E}$ (respectively $\tilde{\Pi}_E$) descends to $\underline{\Pi}_{j,E}$ (respectively $\underline{\tilde{\Pi}}_E$) over $M\times M$. Similarly we denote the descendants of base loci and Bergman metrics by $\underline{B}_{j,E}$, $\underline{\tilde{B}}_E$, $\underline{\Om}_{j,E}$ and $\underline{\tilde{\Om}}_E$.

\begin{definition}
A tower of normal coverings with line bundles $\{(M_j,E)\}$ is Bergman stable if the pull-back Bergman kernels $\{\Pi_{j,E}(p_j(\cdot),p_j(\cdot))\}$ converge locally uniformly to $\tilde{\Pi}_E(\cdot,\cdot)$ over $\tilde{M}\times\tilde{M}$.
\end{definition}

In particular, if $E =K_M$ is the canonical line bundle, the Bergman stability has been studied by many authors (cf. \cite{R} \cite{To} \cite{O}  \cite{CF}  \cite{Ye3}, etc). If one assumes the Bergman stability for $E$, by the standard argument in complex analysis, one can derive the higher order convergence for the Bergman metrics (cf. Proposition \ref{BKCinfty}). Furthermore, we are interested in the equidistribution of the simultaneous zeros of random sections in $H^0(M_j,E)$.

\medskip

Let $d_{j,E} = \dim_{\mathbb{C}} H^0(M_j,E)$ be the complex dimension of the space of holomorphic sections. For any $1\leq l\leq n$, we may consider the Grassmannian of $l$-dimensional complex linear subspaces of $H^0(M_j,E)$, denoted by $\gcal_lH^0(M_j,E)$. Endowing $\gcal_lH^0(M_j,E)$ with the normailzed Haar measure $\mu^{(l)}_{j,E}$, thus we obtain a probability space $\left( \gcal_lH^0(M_j,E),\mu^{(l)}_{j,E} \right)$, of which the expectation is denoted by $\mathbb{E}^{(l)}_j$. Any $\scal^l_{j,E}\in\gcal_lH^0(M_j,E)$ can be written as $\scal^l_{j,E}=\text{Span}\{s_{j_1},\dots,s_{j_l}\}$, where $s_{j_1},\dots,s_{j_l}\in H^0(M_j,E)$ are linearly independent. Let $Z_{\scal^l_{j,E}}\in\dcal'^{l,l}(M_j)$ denote the current of integration over the common zero set of $s_{j_1},...,s_{j_l}$ (to be more specific, it is the current of integration over the regular points of the complex analytic set $\{z\in M_j:\ s_{j_1}(z)=\cdots=s_{j_l}(z)=0\}$), which is independent of the choice of the basis $\{s_{j_1},\dots,s_{j_l}\}$. Whenever $B_{j,E}=\emptyset$, by Bertini's theorem (cf. \cite{GH} pp.137), for a generic (thus almost sure in terms of the probability measure $\mu^{(l)}_{j,E}$) choice of $\scal^l_{j,E}=\text{Span}\{s_{j_1},\dots,s_{j_l}\}$, the zero sets $\{z\in M_j:\ s_{j_k}(z)=0\}$ are smooth and intersect transversely for $k=1, \cdots, l$. Hence $\{z\in M_j:\ s_{j_1}(z)=\cdots=s_{j_l}(z)=0\}$ is a smooth submanifold of $M_j$ with codimension $l$. Therefore, we may ignore multiplicities when considering expectations. In order to work on the same level, we study the normalized zero currents
\begin{align}\label{NMZ}
\underline{Z}_{\scal^l_{j,E}}:=I_j^{-1}{q_j}_*Z_{\scal^l_{j,E}}\in \dcal'^{l,l}(M).
\end{align}

The following is the equidistribution result for a general line bundle $E$.

\begin{prop}\label{EZ}
If the tower of normal coverings with line bundles $\{(M_j,E)\}$ is Bergman stable and $\tilde{B}_E=\emptyset$, then there exists $J\geq0$ such that,
$$\mathbb{E}^{(l)}_j\underline{Z}_{\scal^l_{j,E}}=(\pi^{-1}\underline{\Om}_{j,E})^l$$
for any $j\geq J$ and $1\leq l\leq n$,
as $(l,l)$-currents on $M$. Furthermore,  it satisfies
$$\lim_{j\to\infty}\mathbb{E}^{(l)}_j\underline{Z}_{\scal^l_{j,E}}=(\pi^{-1}\underline{\tilde{\Om}}_E)^l$$
in the sense of currents.
\end{prop}

Next we focus on a positive holomorphic line bundle $(L,h)$ over $M$. Choose the curvature form $\om_h=\frac{\sqrt{-1}}{2}\Theta_h$ as the \kahler\ form of $M$ and $dV_h=\frac{\om_h^n}{n!}$. For any $N\geq1$, the tower of normal coverings with line bundles $\{(M_j,L^N)\}$ can be similarly defined. The main theorem in the paper are the following Bergman stability result for sufficiently positive line bundles over the tower of coverings. The main ingredient in our argument is the theorem of Poincar\'{e} series in \cite{LZ}, from which we can derive the effective estimates of difference between the Bergman kernels on the universal covering and on each level. More precisely, the difference decays exponentially  in terms of a geometric quantity $\tau_j$ (cf.  equation (\ref{tau0})) on the tower of coverings.

\begin{theorem}\label{BK}
There exists $N_1=N_1(M,L,h)>0$ and $\si=\si(M,L,h,N)>0$ when $N\geq N_1$, such that for any compact subsets $K,K'\subset\tilde{M}$, there exists $C_{K,K'}>0$ satisfying
\begin{align*}
\abs{\Pi_{j,L^N}(p_j(z),p_j(w))-\tilde{\Pi}_{L^N}(z,w)}_{h^N}\leq C_{K,K'}e^{-\si\tau_j},
\end{align*}
for all $z\in K,w\in K', N\geq N_1$ and $j$ large enough. 
As a consequence, the tower of normal coverings with line bundles $\{(M_j,L^N)\}$ is Bergman stable for $N\geq N_1$.
\end{theorem}

The following  base point freeness theorem is an application of H\"ormander type $L^2$-estimates.

\begin{theorem}\label{BL}
There exists some $N_2=N_2(M,L,h)>0$ such that for all $N\geq N_2$, the line bundles $L^N\to\tilde{M}$ are $L^2$-base point free, i.e. $\tilde{B}_{L^N}=\emptyset$.
\end{theorem}

A direct consequence is the equidistribution  result of the simultaneous zeros of random sections of the positive line bundle.

\begin{cor}\label{EZ(L)}
Let $N^*(M,L,h)=\max\{N_1,N_2\}$. For all $N\geq N^*$ and $1\leq l\leq n$, the expectation of the normailzed zero current $\underline{Z}_{\scal^l_{j,L^N}}$ satisfies
$$\lim_{j\to\infty}\mathbb{E}^{(l)}_j\underline{Z}_{\scal^l_{j,L^N}}=(\pi^{-1}\underline{\tilde{\Om}}_{L^N})^l$$ in the sense of currents.
\end{cor}

We also consider the unit sphere $SH^0(M_j,L^N)=\{s\in H^0(M_j,L^N):\ \norm{s}_{h^N}=1\}\simeq S^{2d_{j,L^N}-1}$ with the normalized Haar measure $\nu_{j,L^N}$. This probability space is the same as the space $\langle\gcal_1H^0(M_j,L^N),\mu^1_{j,L^N}\rangle$ discussed above. Then the variance of the zero currents can also be estimated in terms of  $\tau_j$. 

\begin{theorem}\label{Variance}
With the same $N^*>0$ as in Corollary \ref{EZ(L)}. For all $N\geq N^*$ and any smooth test form $\psi\in\dcal^{n-1,n-1}(M)$, the normalized zero current $\underline{Z}_{s_j}=I_j^{-1}{q_j}_*Z_{s_j}$ satisfies
\begin{align*}
\begin{split}
Var\left((\underline{Z}_{s_j},\psi)\right)=&\int_{SH^0(M_j,L^N)}\abs{(\underline{Z}_{s_j}-\pi^{-1}\underline{\Om}_{j,L^N},\psi)}^2 d\nu_{j,L^N}(s_j) \\
\lesssim&[\exp\{-c\tau_{\lfloor\frac{j}{2}\rfloor}\}+2^{-\frac{j}{2}}]\ \norm{\sqrt{-1}\pa\bar{\pa}\psi}_{L^1(M)}^2,
\end{split}
\end{align*}
for $j$ large enough, where 
$c=c(M,L,h,N)>0$. In addition,
\begin{align*}
\lim_{j\to\infty}Var\left((\underline{Z}_{s_j},\psi)\right)=0.
\end{align*}
\end{theorem}
We want to point out here that the constants $\si=\si(M,L,h,N)$ in Theorem \ref{BK} and $c=c(M,L,h,N)$ in Theorem \ref{Variance} can be arbitrarily large if we let $N$ be large enough.

\medskip
In \S2, we collect all the preliminaries and  background. In \S3, we 
discuss the equidistribution for
a general holomorphic line bundle $E$. \S4 is devoted to show the Bergman stability 
for a positive holomorphic line bundle $L$. At last, we spend the entire \S5 proving the variance estimates and almost sure convergence of the normalized zero currents.

\subsection*{Acknowledgement}
The authors would like to thank Professor Bernard Shiffman and Professor Steve Zelditch for their helpful discussions and Professor Xiaojun Huang for his constant support. The authors also would like to thank the referee for the penetrating comments.
Part of the work was done when the first author was visiting Capital Normal University in China and the second author was visiting Syracuse University. They are  grateful to  both departments for the warm hospitality.

\section{Preliminaries and Background}

For convenience, we are going to omit the $0$ index in our following notations concerning the base manifold $M=M_0$.
\subsection{Bergman Kernel, Base Locus and Bergman Metric}
The Hermitian inner product of sections of the line bundle $(E,h_E)\to M$ is defined by
$$\llangle s,s'\rrangle:=\int_{M}(s,s')_{h_E}\ dV.$$
If one chooses an orthonormal basis $\{S_k\}_{k=1}^{d_E}$ of $H^0(M,E)$, then for any $z,w\in M$, the Bergman kernel is given by
$$\Pi_E(z,w):=\sum_{k=1}^{d_E}S_k(z)\otimes\overline{S_k(w)}.$$
It is straightforward to check that $\Pi_E(z,w)$ does not depend on the choice of the orthonormal basis and $\Pi_E\in H^0(M\times M,E\boxtimes\bar{E})$ is the integral kernel of the orthogonal projection $L^2(M,E)\to H^0(M,E)$ satisfying the reproducing property:
$$s(w)=\int_{M}(s(z),\Pi_E(z,w))_{h_E}\ dV(z),\text{\quad for all }w\in M\text{ and }s\in H^0(M,E).$$

\medskip

The base locus $B_E$ is the common zero set for all holomorphic sections:
\begin{align}\label{Bergman0}
B_E:=\{z\in M:\ s(z)=0\text{ for all }s\in H^0(M,E)\}=\{z\in M:\ \Pi_E(z,z)=0\}.
\end{align}
Suppose $U,U'\subset M$ are two open sets with local frames $e_E,e'_E$ defined on it. Then there exist holomorphic functions $\{f_k\}_{k=1}^{d_E}\subset\ocal(U)$ and $\{g_k\}_{k=1}^{d_E}\subset\ocal(U')$ such that $S_k=f_ke_E$ over $U$ and $S_k=g_ke'_E$ over $U'$ for $1\leq k\leq d_E$. Hence,
\begin{align}\label{Bergman1}
\Pi_E(z,w)=\Phi_E(z,w)e_E(z)\otimes\overline{e'_E(w)}\text{\quad for }z\in U, w\in U',
\end{align}
where
$$\Phi_E(z,w):=\sum_{k=1}^{d_E} f_k(z)\overline{g_k(w)}$$
is holomorphic in $z\in U$ and anti-holomorphic in $w\in U'$. Moreover, we restrict $\Pi_E$ on the diagonal and denote $$\phi_E(z)=\Phi_E(z,z)=\sum_{k=1}^{d_E}\abs{f_k(z)}^2\text{\quad for }z\in U.$$
Therefore $\phi_E\in\ccal^\infty(U,\R^+)$ and is nonvanishing on $U\setminus B_E$. The Bergman metric $\Om_E$ can be defined on $U\setminus B_E$ by
\begin{align}\label{Bergman2}
\Om_E:=\frac{\sqrt{-1}}{2}\pa\bar{\pa}\log{\phi_E}\geq0,
\end{align}
which is independent of the choice of the local frame $e_E$. So we may choose an open covering of $M$ and $\Om_E$ is thus defined on $M\setminus B_E$.

\medskip


The Bergman kernels $\Pi_{j,E}$ (respectively $L^2$-Bergman kernel $\tilde{\Pi}_E$), base loci $B_{j,E}$ (respectively $L^2$-base locus $\tilde{B}_E$) and Bergman metrics $\Om_{j,E}$ (respectively $\tilde{\Om}_E$) over $M_j$ (respectively $\tilde{M}$) can be defined in a similar way. Since the actions $\Ga/\Ga_j\curvearrowright M_j,\Ga\curvearrowright\tilde{M}$ are given by isometries of the manifolds that preserve the metrics $h_E$ of line bundles, they also preserve the Hermitian inner products of holomorphic sections. Therefore the Bergman kernels are invariant under these actions in the sense that
\begin{align}\label{Bergman3}
\Pi_{j,E}([\ga]_jz,[\ga]_jw)=\Pi_{j,E}(z,w),\text{\quad for all }z,w\in M_j\text{ and }[\ga]_j\in\Ga/\Ga_j,
\end{align}
and
\begin{align}\label{Bergman4}
\tilde{\Pi}_E(\ga z,\ga w)=\tilde{\Pi}_{E}(z,w),\text{\quad for all }z,w\in\tilde{M}\text{ and }\ga\in\Ga.
\end{align}
By (\ref{Bergman0}) and (\ref{Bergman2}), the base loci and Bergman metrics are both induced from the Bergman kernels and they also share the invariant properties. Therefore, we have the descendants $\underline{\Pi}_{j,E}$, $\underline{B}_{j,E}$ and $\underline{\Om}_{j,E}$ on $M$ in the sense that $\underline{\Pi}_{j,E}(q_j(\cdot),q_j(\cdot))=\Pi_{j,E}(\cdot,\cdot)$, $q_j^{-1}(\underline{B}_{j,E})=B_{j,E}$ and $q_j^*\underline{\Om}_{j,E}=\Om_{j,E}$. On the other hand, for any $j\geq0$, as $q_j:M_j\to M$ is a proper local diffeomorphism, the direct image ${q_j}_*\Om_{j,E}$ satisfies
\begin{align}\label{Bergman5}
{q_j}_*\Om_{j,E}=I_j\underline{\Om}_{j,E},
\end{align}
where $I_j = [\Gamma : \Gamma_j] < \infty$ is the covering indices.

\medskip

The following theorem implies that the Bergman kernel on the universal covering of a positive line bundle concentrates on the diagonal. It is generally referred to as Agmon estimates, which serves as a powerful tool in our proofs.
\begin{theorem}[cf. \cite{LZ} Theorem 2.1 or \cite{MM2} Theorem 0.1]
Let $M$ be a complete \kahler\ manifold and $(L,h)\to M$ be a positive holomorphic line bundle. Then there exists some $\be=\be(M,L,h)>0$ such that
the $L^2$-Bergman kernel $\tilde{\Pi}_{L^N}$ of $(L^N,h^N)\to\tilde{M}$ satisfies
\begin{align*}
\abs{\tilde{\Pi}_{L^N}(z,w)}_{h^N}\lesssim e^{-\be\sqrt{N}dist(z,w)},
\end{align*}
for $z,w\in\tilde{M}$ with $dist(z,w)\geq1$.
\end{theorem}

\subsection{Circle Bundle and \szego\ Kernel}

We now focus on a positive holomorphic line bundle $(L,h)$ over $M$, i.e. $h$ is a smooth Hermitian metric with positive curvature form $\om_h=\frac{\sqrt{-1}}{2}\Theta_h = - \frac{\sqrt{-1}}{2} \partial \bar\partial \log h$. $L^{-1}$ denotes its dual bundle with dual metric $h^{-1}$. $\rho$ is a function on $L^{-1}$ given by $\rho(\la):=\abs{\la}^2_{h^{-1}}-1$, which is a defining function for the disc bundle $D:=\{\la\in L^{-1}:\ \abs{\la}_{h^{-1}}\leq1\}$. When $L$ is a positive line bundle, $D$ is a strictly pseudoconvex domain. Therefore the principal $S^1$-bundle $\pi:X\to M$ given by $X:=\{\la\in L^{-1}:\abs{\la}_{h^{-1}}=1\}=\pa D$ is a strictly pseudoconvex $CR$ manifold. 
$\al:=-\sqrt{-1}\pa\rho|_X=\sqrt{-1}\bar{\pa}\rho|_X$ is a contact form on $X$ with $d\al=2\pi^*\om_h$ and $dV_X:=\frac{\al\wedge(d\al)^n}{2^{n+1}\pi n!}$ is a volume form on $X$. For any $N\geq1$, we can lift a section $s\in H^0(M,L^N)$ to $\hat{s}\in\hcal^2_N(X)$, the Hardy space of $L^2$ CR-functions on $X$ satisfying the equivariant condition $\hat{s}(r_\theta x)=e^{\sqrt{-1}N\theta}\hat{s}(x)$, where $x\in X$ and $r_{\theta}$ denotes the $S^1$-action on $X$. In fact,
$$\hat{s}(x):=\langle x^N,s(\pi(x))\rangle,$$
where $\langle\cdot\rangle$ denotes the pairing of $L^N$ with $L^{-N}$. If $e_L$ is a local frame of $L$ over some open set $U$, for $z\in U$, we use $(z,\theta)$ as local coordinate for $x=e^{\sqrt{-1}\theta}\abs{e_L(z)}_he^{-1}_L(z)\in X$. Suppose that $s=fe^N_L$ over $U$ for some $f\in\ocal(U)$,
then in terms of the local coordinates,
$$\hat{s}(z,\theta)=\langle(e^{\sqrt{-1}\theta}\abs{e_L(z)}_he^{-1}_L(z))^N,f(z)e^N_L(z)\rangle=e^{\sqrt{-1}N\theta}\abs{e_L(z)}_h^Nf(z).$$
As a result, the lifting preserves the $L^2$-inner products:
\begin{align}\label{Equivariant}
\llangle s,s'\rrangle=(\hat{s},\hat{s}')_{L^2(dV_X)}:=\int_X\hat{s}\bar{\hat{s}}'\ dV_X.
\end{align}
Let $\{S_k\}_{k=1}^{d_{L^N}}$ be an orthonormal basis of $H^0(M,L^N)$. Then by (\ref{Equivariant}), $\{\hat{S}_k\}_{k=1}^{d_{L^N}}$ forms an orthonormal basis of $\hcal^2_N(X)$. In this way, the Bergman kernel $\Pi_{L^N}$ can be lifted to the \szego\ kernel of $\hcal^2_N(X)$: $$\hat{\Pi}_N(x,y)=\sum_{k=1}^{d_{L^N}}\hat{S}_k(x)\overline{\hat{S}_k(y)}\text{\quad for }x,y\in X.$$
Similarly we can define circle bundles $\pi_j:X_j\to M_j$, $\tilde{\pi}:\tilde{X}\to\tilde{M}$ and \szego\ kernels $\hat{\Pi}_{j,N}$, $\hat{\tilde{\Pi}}_N$ of $X_j$ and $\tilde X$ respectively, of which the local expressions are
\begin{align}\label{Szego1}
\hat{\Pi}_{j,N}(z,\theta,w,\varphi)=e^{\sqrt{-1}N(\theta-\varphi)}\abs{e_{j,L}(z)}_h^N\abs{e'_{j,L}(w)}_h^N\Phi_{j,L^N}(z,w),\text{\quad for }j\geq0,
\end{align}
if $\Pi_{j,L^N}(z,w)=\Phi_{j,L^N}(z,w)e^N_{j,L}(z)\otimes\overline{e'^N_{j,L}(w)}$, and
\begin{align}\label{Szego2}
\hat{\tilde{\Pi}}_N(z,\theta,w,\varphi)=e^{\sqrt{-1}N(\theta-\varphi)}\abs{\tilde{e}_L(z)}_h^N\abs{\tilde{e}'_L(w)}_h^N\tilde{\Phi}_{L^N}(z,w),
\end{align}
if $\tilde{\Pi}_{L^N}(z,w)=\tilde{\Phi}_{L^N}(z,w)\tilde{e}^N_L(z)\otimes\overline{\tilde{e}'^N_L(w)}$. Therefore,
\begin{align}\label{Szego3}
\abs{\hat{\tilde{\Pi}}_N(x,y)}=\abs{\tilde{\Pi}_{L^N}(\tilde{\pi}(x),\tilde{\pi}(y))}_{h^N},\text{\quad for all }x,y\in\tilde{X}.
\end{align}

\medskip

The action $\Ga\curvearrowright\tilde{M}$ can be lifted as a group of CR holomorphic contact transformations on $\tilde{X}$ preserving the contact form $\tilde{\al}$. To be more specific, in terms of compatible local coordinates on $\tilde{X}$ (i.e. if we take $\tilde{e}_L$ as a local frame of $L\to \tilde{M}$ near $z\in\tilde{M}$, then we will take $\tilde{e}_L\circ\ga^{-1}$ as a local frame near $\ga z$),
\begin{align}\label{Szego4}
\ga(z,\theta)=(\ga z,\theta).
\end{align}
Hence the action $\Ga\curvearrowright\tilde{X}$ commutes with the $S^1$-action. As the Bergman kernel $\tilde{\Pi}_{L^N}$ satisfies (\ref{Bergman4}), by (\ref{Szego2}) and (\ref{Szego4}),
\begin{align*}
\hat{\tilde{\Pi}}_N(\ga x,\ga y)=\hat{\tilde{\Pi}}_N(x,y)\text{\quad for all }x,y\in\tilde{X}\text{ and }\ga\in\Ga.
\end{align*}
For each $j\geq0$, the covering map $p_j:\tilde{M}\to M_j$ induces a map $\hat{p}_j:\tilde{X}\to X_j$ such that the following diagram commutes:
\begin{align*}
\begin{array}{ccccc}
  & \tilde{M} & \stackrel{\tilde{\pi}}\longleftarrow & \tilde{X} & \\
  p_j & \downarrow &  & \downarrow & \hat{p}_j \\
  & M_j & \stackrel{\pi_j}\longleftarrow & X_j &
\end{array}.
\end{align*}
In fact, under compatible local coordinates (i.e. for any $z\in \tilde{M}$, if we take $e_{j,L}$ as a local frame of $L\to M_j$ near $p_j(z)\in M_j$, then we will take $\tilde{e}_L=e_{j,L}\circ p_j$ as a local frame of $L\to\tilde{M}$ near $z$),
\begin{align}\label{Szego5}
\hat{p}_j(z,\theta)=(p_j(z),\theta).
\end{align}

The following theorem proved by Z. Lu and S. Zelditch describes the relation between the \szego\ kernels over a manifold and those over the universal covering, which is the essential ingredient in our proof.

\begin{theorem}[\cite{LZ} Theorem 1]
There exists $N_0=N_0(M,L,h)>0$ such that if $N\geq N_0$, then for all $j\geq0$,
\begin{align}\label{LZThm}
\hat\Pi_{j,N}(\hat{p}_j(x),\hat{p}_j(y))=\di\sum_{\ga_j\in\Ga_j}\hat{\tilde{\Pi}}_N(\ga_j x,y),\text{\quad for any }x,y\in\tilde{X}.
\end{align}
\end{theorem}


\section{Equidistribution for a General Line Bundle $E$}

The following proposition asserts that Bergman stability implies higher order convergence of the Bergman kernels. It follows from the standard normal family argument (cf. Proposition 3.5 in \cite{To}).
\begin{prop}\label{BKCinfty}
If the tower of normal coverings with line bundles $\{(M_j,E)\}$ is Bergman stable, then the pull-back Bergman kernels $\{\Pi_{j,E}(p_j(\cdot),p_j(\cdot))\}$ converge locally uniformly in $\ccal^\infty$ topology to $\tilde{\Pi}_E(\cdot,\cdot)$ over $\tilde{M}\times\tilde{M}$.
\end{prop}

\begin{proof}
Let $(U,V),(U',V')$ be any two pairs of bounded open sets in $\tilde{M}$ such that:
\begin{enumerate}[a.]
\item $V\subset\subset U$ and $V'\subset\subset U'$.
\item The restrictions $p_0|_U$ and $p_0|_{U'}$ are one-to-one, which implies that $p_j|_U$ and $p_j|_{U'}$ are one-to-one for any $j\geq0$.
\item$U$ is contained in the domain of a local frame $\tilde{e}_E$ of $E\to\tilde{M}$ as well as a holomorphic coordinate system $\{\xi=(\xi_1,\dots,\xi_n)\}$, while $U'$ is contained in the domain of a local frame $\tilde{e}'_E$ as well as a holomorphic coordinate system $\{\eta=(\eta_1,\dots,\eta_n)\}$.
\end{enumerate}
Hence for all $j\geq0$, we may define $e_{j,E}:=\tilde{e}_{E}\circ p^{-1}_j$ and $e'_{j,E}:=\tilde{e}'_{E}\circ p^{-1}_j$ as local frames of $E\to M_j$ over $p_j(U)$ and $p_j(U')$ respectively. Then as in (\ref{Bergman1}),
$$\tilde{\Pi}_E(z,w)=\tilde{\Phi}_E(z,w)\tilde{e}_{E}(z)\otimes\overline{\tilde{e}_E(w)},$$
and for $j\geq0$,
$$\Pi_{j,E}(p_j(z),p_j(w))=\Phi_{j,E}(p_j(z),p_j(w))e_{j,E}(p_j(z))\otimes\overline{e'_{j,E}(p_j(w))}=\Phi^*_{j,E}(z,w)\tilde{e}_{E}(z)\otimes\overline{\tilde{e}_E(w)},$$
where $\Phi^*_{j,E}(z,w):=\Phi_{j,E}(p_j(z),p_j(w))$ is also holomorphic in $z\in U$ and antiholomorphic in $w\in U'$. Take $t,t'>0$ such that for any $\xi\in V$ and $\eta\in V'$ , the coordinate polydiscs $\Pi_{i=1}^n\bar{D}(\xi_i,t)\subset U$ and $\Pi_{i=1}^n\bar{D}(\eta_i,t')\subset U'$. For arbitrary multi-indices $\al=(\al_1,\dots,\al_n)$ and $\be=(\be_1,\dots,\be_n)$, using Cauchy's integral formula, we have that for any $\xi^0\in V$ and $\eta^0\in V'$,
\begin{align*}
\begin{split}
&D^{\al}_{\xi}\bar{D}^{\be}_{\eta}\Phi^*_{j,E}(\xi^0,\eta^0) \\
=&\frac{\al!\be!}{(2\pi)^{2n}}\int_{\abs{\xi_1-\xi^0_1}=t}\cdots\int_{\abs{\xi_n-\xi^0_n}=t}\int_{\abs{\eta_1-\eta^0_1}=t'}\cdots\int_{\abs{\eta_n-\eta^0_n}=t'}
\frac{\Phi^*_{j,E}(\xi,\eta)}{\Pi_{i=1}^n(\xi_i-\xi^0_i)^{\al_i+1}(\bar{\eta}_i-\bar{\eta}^0_i)^{\be_i+1}}\ d\xi d\bar{\eta},
\end{split}
\end{align*}
which implies that
\begin{align*}
\abs{D^{\al}_{\xi}\bar{D}^{\be}_{\eta}\Phi^*_{j,E}(\xi^0,\eta^0)}\leq\frac{\al!\be!}{t^{\abs{\al}}t'^{\abs{\be}}}\norm{\Phi^*_{j,E}}_{\ccal^0(U\times U')}.
\end{align*}
Hence for any $k>0$, there exists a constant $C=C(U,V,U',V',k)>0$ such that
\begin{align*}
\norm{\Phi^*_{j,E}}_{\ccal^k(V\times V')}\leq C\norm{\Phi^*_{j,E}}_{\ccal^0(U\times U')}.
\end{align*}
Moreover, the Bergman stability assumption implies that $\left\{\Phi^*_{j,E}\right\}$ converges uniformly on $\bar{U}\times\bar{U}'$ to $\tilde{\Phi}_E$. Thus, for $j$ sufficiently large, we have
\begin{align}\label{BKCinfty1}
\norm{\Phi^*_{j,E}}_{\ccal^k(V\times V')}\leq C(\norm{\tilde{\Phi}_E}_{\ccal^0(U\times U')}+1).
\end{align}
To prove the locally uniform $\ccal^1$ convergence of the Bergman kernels, it suffices to show that the derivatives of the sequence $\left\{\Phi^*_{j,E}\right\}$ converge locally uniformly to those of $\tilde{\Phi}_E$ on $U\times U'$. If not, by taking a subsequence if necessary, we may assume that there exist some $1\leq i\leq n$, compact sets $K\subset U, K'\subset U'$ and $\ep>0$ such that
\begin{align}\label{BKCinfty2}
\sup_{K\times K'}\Abs{\frac{\pa\Phi^*_{j,E}}{\pa\xi_i}-\frac{\pa\tilde{\Phi}_E}{\pa\xi_i}}\geq\ep\text{ \quad for all }j\geq0.
\end{align}
However, since $\{\frac{\pa\Phi^*_{j,E}}{\pa\xi_i}\}$ and their derivatives are uniformly bounded on $K\times K'$ by (\ref{BKCinfty1}), applying Arzel$\grave{a}$-Ascoli theorem, we then have a subsequence $\{\frac{\pa\Phi^*_{j_s,E}}{\pa\xi_i}\}$ converges uniformly on $K\times K'$. As $\left\{\Phi^*_{j_s,E}\right\}$ converges uniformly to $\tilde{\Phi}_E$ on $K\times K'$, $\{\frac{\pa\Phi^*_{j_s,E}}{\pa\xi_i}\}$ must converge uniformly to $\frac{\pa\tilde{\Phi}_E}{\pa\xi_i}$, which contradicts to (\ref{BKCinfty2}). Thus we prove the locally uniform $\ccal^1$ convergence of the Bergman kernels. The higher order convergence follows by induction.
\end{proof}

\begin{prop}\label{BM}
If the tower of normal coverings with line bundles $\{(M_j,E)\}$ is Bergman stable and $\tilde{B}_E=\emptyset$, then there exists $J\geq0$ such that for $j\geq J$, the base loci $B_{j,E}=\emptyset$ and the Bergman metrics $\Om_{j,E}$ can be defined all over $M_j$. Moreover, $\{\underline{\Om}_{j,E}\}_{j=J}^\infty$ converges to $\underline{\tilde{\Om}}_E$ uniformly in $\ccal^\infty$ topology.
\end{prop}

\begin{proof}
Let$D_0\subset\subset\tilde{M}$ be a fundamental domain corresponding to $M_0=M$, i.e. it satisfies $p_0|_{D_0}:D_0\to M$ is injective and $p_0|_{\bar{D}_0}:\bar{D}_0\to M$ is surjective. Since $\{(M_j,E)\}$ is Bergman stable and $\tilde{\Pi}_E$ is nonvanishing on the diagonal, for the compact set $\bar{D}_0\subset\tilde{M}$, there exists $J\geq0$ such that for any $j\geq J$, $\Pi_{j,E}(p_j(z),p_j(z))\neq0$ for $z\in\bar{D}_0$. Hence $B_{j,E}\cap p_j(\bar{D}_0)=\emptyset$ for $j\geq J$ and thus $\underline{B}_{j,E}\cap p_0(\bar{D}_0)=q_j(B_{j,E})\cap q_jp_j(\bar{D}_0)=\emptyset$. Since $p_0(\bar{D}_0)=M$, $\underline{B}_{j,E}=\emptyset$. Therefore $B_{j,E}=q_j^{-1}(\underline{B}_{j,E})=\emptyset$ for $j\geq J$. The remaining part of this proposition follows from the definition of Bergman metric (\ref{Bergman2}) and Proposition \ref{BKCinfty}.
\end{proof}

Proposition \ref{EZ} then follows from the standard arguments as in \cite{SZ1}.

\begin{proof}[Proof of Proposition \ref{EZ}]
The first part of this statement follows from Proposition \ref{BM} and Lemma 4.3 in \cite{SZ1}. The second part also follows from Proposition \ref{BM}.
\end{proof}

\section{Positive Line Bundle $L$ over a Tower of Coverings}

The following geometric quantity describes the profound geometry of the tower of coverings, which was first appeared in \cite{DW}.

\begin{definition}\label{tau}
For any $j\geq0$,
\begin{align}\label{tau0}
\tau_j=\inf\left\{dist(z,\gamma_jz):z\in\tilde{M},\gamma_j\in\Gamma_j\setminus\{1\}\right\}.
\end{align}
\end{definition}

One can check that $\tau_j\geq2R_j$, where $R_j$ denotes the injectivity radius of $M_j$. It is easy to see that for all $z\in\tilde{M}$, $p_j|_{B(z,\half\tau_j)}$ is one-to-one, where $B(z,\half\tau_j)$ denotes the geodesic ball in $\tilde{M}$ centered at $z$ of radius $\half\tau_j$.
For all $j\geq0$, $\Gamma_j\setminus\{1\}\supset\Gamma_{j+1}\setminus\{1\}$. Hence the sequence $\{\tau_j\}$ is nondecreasing. The following lemma describes the growth of $\tau_j$, which was obtained in \cite{DW} \cite{Don1}. 

\begin{lem}\label{tau}
$\di\lim_{j\to\infty}\tau_j=\infty$.
\end{lem}

\begin{proof}
Argue by contradiction. We assume that there exists $C>0$ such that for all $j\geq0$, $\tau_j\leq C$. Then for all $j\geq0$, there exist $z_j\in\tilde{M}$ and $\gamma_j\in\Gamma_j\setminus\{1\}$ with $dist(z_j,\gamma_jz_j)\leq2C$. Also we know that $p_0|_{\bar{D}_0}$ is surjective, where $D_0\subset\subset\tilde{M}$ is the fundamental domain described in the proof of Proposition \ref{BM}, and for all $j\geq0$, $p_0^{-1}(p_0(z_j))=\Gamma z_j$. Hence for all $j$, there exists $g_j\in\Gamma$ such that $g_jz_j\in\bar{D}_0$. Denote $z'_j=g_jz_j\in\bar{D}_0$ and $\gamma'_j=g_j\gamma_jg_j^{-1}\in\Gamma_j\setminus\{1\}$ since $\Gamma_j$ is a normal subgroup of $\Gamma$, we then have $dist(z'_j,\gamma'_jz'_j)=dist(z_j,\gamma_jz_j)\leq2C$ as $g_j\in\Gamma$ is an isometry on $\tilde{M}$. By the compactness of $\bar{D}_0$, there exists a subsequence $\{j_k\}$ such that $z'_{j_k}\to z^*\in\bar{D}_0$. Since $dist(z^*,\gamma'_{j_k}z^*)\leq dist(z^*,z'_{j_k})+dist(z'_{j_k},\gamma'_{j_k}z'_{j_k})+dist(\gamma'_{j_k}z'_{j_k},\gamma'_{j_k}z^*)\leq2dist(z^*,z'_{j_k})+2C$, there exists some $K>0$ with $\{\gamma'_{j_k}z^*\}_{k=K}^{\infty}\subset\bar{B}(z^*,3C)$. Choosing a subsequence again if necessary, we may assume $\gamma'_{j_k}z^*\to w\in\bar{B}(z^*,3C)$. Thus $p_0(z^*)=p_0(\gamma'_{j_k}z^*)\to p_0(w)$, i.e. $p_0(z^*)=p_0(w)$, which implies that there exists some $h\in\Gamma$ such that $z^*=hw$. Hence $\gamma'_{j_k}hw\to w$. Since the group action $\Gamma\curvearrowright\tilde{M}$ is properly discontinuous, $\gamma'_{j_k}h=1$ for $k$ large enough. Therefore $h\in\Gamma_{j_k}$ for $k$ large. But we know $\bigcap_{j=0}^\infty\Ga_j=\{1\}$, so $h=1$. Thus $\gamma'_{j_k}=1$ for $k$ large, which draws a contradiction since $\gamma'_{j_k}\in\Gamma_{j_k}\setminus\{1\}$.
\end{proof}

Hence we can assume that $\tau_0\geq2$ and now begin the proof of Theorem \ref{BK}



\begin{proof}[Proof of Theorem \ref{BK}]
Similar as in the proof of Proposition \ref{BKCinfty}, we can take bounded open sets $U, U'\subset\tilde{M}$ satisfying conditions $(b)$ and $(c)$, local frames $\tilde{e}_L$ and $\tilde{e}'_L$ over $U$ and $U'$, $e_{j,L}=\tilde{e}_L\circ (p_j|_U)^{-1}$ and $e'_{j,L}=\tilde{e}'_L\circ (p_j|_{U'})^{-1}$ over $p_j(U)$ and $p_j(U')$. Under these local frames, we adopt similar notations to write: for any $z\in U, w\in U'$,
\begin{align*}
\begin{split}
\Pi_{j,L^N}(p_j(z),p_j(w))=&\Phi_{j,L^N}(p_j(z),p_j(w))e^N_{j,L}(p_j(z))\otimes\overline{e'^N_{j,L}(p_j(w))} \\
=&\Phi_{j,L^N}(p_j(z),p_j(w))\tilde{e}^N_L(z)\otimes\overline{\tilde{e}'^N_L(w)},
\end{split}
\end{align*}
and
\begin{align*}
\tilde{\Pi}_{L^N}(z,w)=\tilde{\Phi}_{L^N}(z,w)\tilde{e}^N_L(z)\otimes\overline{\tilde{e}'^N_L(w)}.
\end{align*}
Hence, by (\ref{Szego1}) and (\ref{Szego2}),
\begin{align*}
\abs{\Pi_{j,L^N}(p_j(z),p_j(w))-\tilde{\Pi}_{L^N}(z,w)}_{h^N}=\abs{\hat\Pi_{j,N}(p_j(z),\theta,p_j(w),\varphi)-\hat{\tilde{\Pi}}_N(z,\theta,w,\varphi)},
\end{align*}
for any $\theta,\varphi\in[0,2\pi]$. On the other hand, take any $N\geq N_0$, then (\ref{LZThm}) holds for all $j\geq0$. By triangle inequality, we have: for all $x,y\in\tilde{X}$,
\begin{align*}
\abs{\hat\Pi_{j,N}(\hat{p}_j(x),\hat{p}_j(y))-\hat{\tilde{\Pi}}_N(x,y)}\leq\sum_{\ga_j\in\Ga_j\setminus\{1\}}\abs{\hat{\tilde{\Pi}}_N(\ga_jx,y)}.
\end{align*}
Take any $z\in U,w\in U'$ and let $x=(z,\theta),y=(w,\varphi)\in\tilde{X}$ in terms of the local frames $\tilde{e}_L,\tilde{e}'_L$, by (\ref{Szego3}), (\ref{Szego4}) and (\ref{Szego5}), the inequality above implies
\begin{align*}
\abs{\hat\Pi_{j,N}(p_j(z),\theta,p_j(w),\varphi)-\hat{\tilde{\Pi}}_N(z,\theta,w,\varphi)}
\leq\sum_{\ga_j\in\Ga_j\setminus\{1\}}\abs{\tilde{\Pi}_{L^N}(\ga_jz,w)}_{h^N}.
\end{align*}
Hence we obtain
\begin{align*}
\abs{\Pi_{j,L^N}(p_j(z),p_j(w))-\tilde{\Pi}_{L^N}(z,w)}_{h^N}\leq\sum_{\ga_j\in\Ga_j\setminus\{1\}}\abs{\tilde{\Pi}_{L^N}(\ga_jz,w)}_{h^N}.
\end{align*}
As $\{\tau_j\}$ increases to $\infty$, for any compact sets $K\subset U,K'\subset U'$, there exists $a_{K,K'}\in\N$ such that
$$\half\tau_{a_{K,K'}}\geq\sup_{z\in K,w\in K'}dist(z,w).$$
Thus for any $j\geq a_{K,K'}$, $z\in K,w\in K'$ and $\ga_j\in\Ga_j\setminus\{1\}$,
\begin{align*}
dist(\ga_jz,w)\geq dist(\ga_jz,z)-dist(z,w)\geq\tau_j-\half\tau_{a_{K,K'}}\geq\half\tau_j\geq1.
\end{align*}
Then by Agmon estimates, there exists $\beta=\beta(M,L,h)>0$ such that
\begin{align*}
\begin{split}
&\sum_{\ga_j\in\Ga_j\setminus\{1\}}\abs{\tilde{\Pi}_{L^N}(\ga_jz,w)}_{h^N} \\
\leq&\sum_{k=0}^{\infty}\sharp\left\{\ga_j\in\Ga_j\setminus\{1\}:(k+\half)\tau_j\leq dist(\ga_jz,w)<(k+\frac{3}{2})\tau_j\right\}e^{-\be\sqrt{N}(k+\half)\tau_j} \\
\leq&\sum_{k=0}^{\infty}\sharp\left\{\ga_j\in\Ga_j:\ga_jz\in B(w,(k+\frac{3}{2})\tau_j)\right\}e^{-\be\sqrt{N}(k+\half)\tau_j}.
\end{split}
\end{align*}
Whenever $\ga_jz\in B(w,(k+\frac{3}{2})\tau_j)$,
$$B(\ga_jz,\half\tau_j)\subset B(w,(k+2)\tau_j).$$
Also from the definition of $\tau_j$,
$$B(\ga z,\half\tau_j)\cap B(\ga'z,\half\tau_j)=\emptyset\text{\quad for }\ga,\ga'\in\Ga_j,\ \ga\neq\ga'.$$
Furthermore, since $\Ga$ acts isometrically on $\tilde{M}$,
$$V(B(\ga z,\half\tau_j))=V(B(z,\half\tau_j)), \text{\quad for all }\ga\in\Ga.$$
Hence for any $j\geq a_{K,K'}$,
\begin{align*}
\begin{split}
&\sharp\left\{\ga_j\in\Ga_j:\ga_jz\in B(w,(k+\frac{3}{2})\tau_j)\right\} \\
&\leq\frac{V(B(w,(k+2)\tau_j))}{V(B(z,\half\tau_j))} \\
&\leq\frac{V(B(w,(k+2)\tau_j))}{V(B(z,\half\tau_{a_{K,K'}}))} \\
&=\frac{V(B(w,\half\tau_{a_{K,K'}}))}{V(B(z,\half\tau_{a_{K,K'}}))}\frac{V(B(w,(k+2)\tau_j))}{V(B(w,\half\tau_{a_{K,K'}}))}.
\end{split}
\end{align*}
Since Ricci curvature of the compact manifold $M$ is bounded, there exists $K>0$ such that $Ric(M)\geq-(2n-1)K$. Therefore $Ric(\tilde{M})\geq-(2n-1)K$. By Bishop-Gromov volume comparison theorem (cf. \cite{SY} pp.11), if $V(2n,-K,R)$ denotes the volume of the geodesic balls of radius R in the space form $M^{2n}_{-K}$ of constant sectional curvature $-K$, then for all $j\geq a_{K,K'}$ and $k\geq0$,
\begin{align*}
\frac{V(B(w,(k+2)\tau_j))}{V(2n,-K,(k+2)\tau_j)}\leq\frac{V(B(w,\half\tau_{a_{K,K'}}))}{V(2n,-K,\half\tau_{a_{K,K'}})} \
\Rightarrow\ \frac{V(B(w,(k+2)\tau_j))}{V(B(w,\half\tau_{a_{K,K'}}))}\leq\frac{V(2n,-K,(k+2)\tau_j)}{V(2n,-K,\half\tau_{a_{K,K'}})}.
\end{align*}
We may rescale the metric on the space form $M_{-K}^{2n}$ by $K$ and get the space form $M_{-1}^{2n}$, while the volume ratio remains unchanged. Hence
\begin{align*}
\frac{V(B(w,(k+2)\tau_j))}{V(B(w,\half\tau_{a_{K,K'}}))}\leq\frac{V(2n,-1,(k+2)\tau_j)}{V(2n,-1,\half\tau_{a_{K,K'}})},
\end{align*}
\begin{align*}
\Rightarrow\sharp\left\{\ga_j\in\Ga_j:\ga_jz\in B(w,(k+\frac{3}{2})\tau_j)\right\}\leq\frac{V(B(w,\half\tau_{a_{K,K'}}))}{V(B(z,\half\tau_{a_{K,K'}}))}\frac{V(2n,-1,(k+2)\tau_j)}{V(2n,-1,\half\tau_{a_{K,K'}})}.
\end{align*}
For all $R>0$, by an explicit formula (cf.\cite{SY} pp.9),
\begin{align*}
\begin{split}
V(2n,-1,R)&=\si_{2n-1}\int_0^R(\sinh t)^{2n-1}\ dt \\
&=\si_{2n-1}\int_0^R(\cosh^2t-1)^{n-1}\ d(\cosh t) \\
&=\si_{2n-1}\int_1^{\cosh R}(u^2-1)^{n-1}\ du \\
&\leq\si_{2n-1}\int_0^{e^R}u^{2n-2}\ du \\
&=\frac{\si_{2n-1}}{2n-1}e^{(2n-1)R},
\end{split}
\end{align*}
here $\si_{2n-1}$ denotes the Euclidean volume of the unit sphere $S^{2n-1}\subset\R^{2n}$. Since $K,K'\subset\tilde{M}$ are compact, there exists a constant $\tilde{C}_{K,K'}>0$ depending on $K,K'$, such that
\begin{align*}
\frac{V(B(w,\half\tau_{a_{K,K'}}))}{V(B(z,\half\tau_{a_{K,K'}}))}\leq\tilde{C}_{K,K'},\text{ for all }z\in K,w\in K'.
\end{align*}
Take $\hat{C}_{K,K'}=\frac{\si_{2n-1}\tilde{C}_{K,K'}}{(2n-1)V(2n,-1,\half\tau_{a_{K,K'}})}$, then
\begin{align*}
\sharp\left\{\ga_j\in\Ga_j:\ga_jz\in B(w,(k+\frac{3}{2})\tau_j)\right\}\leq\hat{C}_{K,K'}e^{(2n-1)(k+2)\tau_j},
\end{align*}
for all $z\in K,w\in K'$ and $j\geq a_{K,K'}$.
Therefore,
\begin{align*}
\begin{split}
&\abs{\Pi_{j,L^N}(p_j(z),p_j(w))-\tilde{\Pi}_{L^N}(z,w)}_{h^N} \\
\leq&\hat{C}_{K,K'}\sum_{k=0}^{\infty}e^{(2n-1)(k+2)\tau_j}e^{-\be\sqrt{N}(k+\half)\tau_j} \\
=&\hat{C}_{K,K'}e^{(4n-2-\half\be\sqrt{N})\tau_j}\sum_{k=0}^{\infty}(e^{(2n-1-\be\sqrt{N})\tau_j})^k.
\end{split}
\end{align*}
Denote $N_1=N_1(M,L,h)=\max\{\lfloor(\frac{8n-2}{\be(M,L,h)})^2\rfloor+1,N_0(M,L,h)\}$. Hence when $N\geq N_1$, $\si=\si(M,L,h,N)=-(4n-2-\half\be\sqrt{N})\geq1$ and $2n-1-\be\sqrt{N}\leq-\half\si$. For all $z\in K,w\in K'$ and $j\geq a_{K,K'}$,
\begin{align*}
\abs{\Pi_{j,L^N}(p_j(z),p_j(w))-\tilde{\Pi}_{L^N}(z,w)}_{h^N}\leq\hat{C}_{K,K'}\frac{e^{-\si\tau_j}}{1-e^{-\half\si\tau_j}}\leq\frac{\hat{C}_{K,K'}}{1-e^{-\si}}e^{-\si\tau_j},
\end{align*}
since $\tau_j\geq\tau_0\geq2$. Take $C_{K,K'}=\frac{\hat{C}_{K,K'}}{1-e^{-\si}}$ and we get the desired estimate. Then Bergman stability follows from Lemma \ref{tau}.
\end{proof}



\medskip

\begin{rmk}
Let $(M, \omega)$ be a compact symplectic manifold of real dimension $2n$ with $\frac{\omega}{\pi}$ an integral cohomology class, $L\rightarrow M$ be a Hermitian line bundle such that the curvature $\om_h=\frac{\sqrt{-1}}{2}\Theta_h = \omega$ and $J$ be an almost complex structure compatible with $\omega$. Suppose that $(M, \omega)$ admits a tower of normal coverings. Denote the Bergman kernel of $D_j$ (respectively $\tilde D$), i.e. the Schwartz kernel for the orthogonal projection from $L^2_J(M_j, L^N)$ (respectively $L^2_J(\tilde M, L^N)$) onto the ($L^2$) kernel of $D_j$ (respectively, of $\tilde D$) on $M_j$ (respectively, $\tilde M$), by $\Pi_{j,L^N}(\cdot, \cdot)$ (respectively $\tilde\Pi_{L^N}(\cdot, \cdot)$), where $D_j$ (respectively, $\tilde D$) is the pseudodifferential operator so that $D_j \Pi_{j, L^N}=0$ (respectively $\tilde D \tilde\Pi_{L^N}=0$). (cf. \cite{MM1} \cite{SZ2})
Then by applying Theorem 0.1 and 0.2 in \cite{MM2} and the same argument as above, the analogue conclusion to Theorem \ref{BK} in the symplectic geometric setting also holds. 

\end{rmk}

\begin{rmk}
As pointed out by the referee, if the tensor power $N$ of the line bundle $L$ is sufficiently large, then a tower of coverings (not necessarily normal) with line bundles $(M_j, L^N)$ is Bergman stable. The proof follows from the heat kernel argument in \cite{Don2} (cf. section 1). In Donnelly's proof, the assumption is   that
the smallest nonzero eigenvalue of the Hodge-Laplacian is uniformly bounded below by a positive constant (independent of $j$), while this assumption is always true if $N$ is sufficiently large by the Bochner-Kodaira identity. On the other hand, with the assumption of the normal coverings, the Bergman stability without the effective estimates follows also from the standard H\"ormander type $L^2$-estimates and the estimate of $\tau_j$ (cf. Proposition \ref{tau}), which we attach in the Appendix.
\end{rmk}

\medskip

Theorem \ref{BL} follows from a very similar argument of $L^2$-estimates as in the proof of second part of Proposition \ref{bsne} in Appendix. Hence we skip the proof here.

\begin{rmk}
Since $\tilde{B}_{L^N}=\emptyset$ for $N\geq N_2$, Proposition \ref{BM} implies that $B_{j,L^N}=\emptyset$ for $j$ sufficiently large. In fact, as $L\to M$ is an ample line bundle, we may choose some $N_3=N_3(M,L)>0$ so that $L^N\to M$ is very ample for all $N\geq N_3$, hence base point free. By pulling back holomorphic sections from $M$ to $M_j$, we are able to show that $B_{j,L^N}=\emptyset$ for all $j\geq0$ if $N\geq N_3$.
\end{rmk}

\begin{proof}[Proof of Corollary \ref{EZ(L)}]
This follows directly from Proposition \ref{EZ}, Theorem \ref{BK} and Theorem \ref{BL}.
\end{proof}

\section{Variance Estimate and Almost Sure Convergence}

In this section, we will derive the variance estimate. The essential ingredient is still the theorem of Poincar\'{e} series in \cite{LZ}. We also rely on  the deep explicit formula for the variance in \cite{SZ4}.

\begin{proof}[Proof of Theorem \ref{Variance}]
We only consider those $j\geq J$, where $J$ is mentioned in Proposition \ref{BM} for $E=L^N$. Taking a partition of unity if necessary, we may assume that $\supp(\psi)\subset U$ for some open set $U\subset M$, which is the domain of some local frame $e_L$ of $L\to M$. Then $e_{j,L}:=e_L\circ q_j$ is a local frame of $L\to M_j$ over $q_j^{-1}(U)$. Moreover, by making $U$ even smaller, it is also possible to assume that $p^{-1}_0(U)$ is the disjoint union of $\ga \tilde{U}$'s for all $\ga\in\Ga$, where $\tilde{U}\subset\tilde{M}$ is such that $p_0|_{\tilde{U}}$ is one-to-one. Denote $p_j(\tilde{U})=U_j$. Hence $q^{-1}_j(U)$ is the union of $[\ga]_jU_j$'s for all $[\ga]_j\in\Ga/\Ga_j$ and $q_j|_{U_j}$ is one-to-one. Thereafter, for $z,w\in U$, we would use $z_j,w_j\in U_j$ and $\tilde{z},\tilde{w}\in\tilde{U}$ to denote their preimages. Choosing an orthonormal basis $\{S_{j_k}\}_{k=1}^{d_{j,L^N}}$ of $H^0(M_j,L^N)$, we assume that for $1\leq k\leq d_{j,L^N}$, $S_{j_k}=f_{j_k}e^N_{j,L}$ over $q_j^{-1}(U)$ for some $f_{j_k}\in\ocal(q_j^{-1}(U))$. Write $f_j=(f_{j_1},\dots,f_{j_{d_{j,L^N}}})$. Hence $\sqrt{-1}\pa\bar{\pa}\log\abs{f_j}=\Om_{j,L^N}$ over $q_j^{-1}(U)$ when $j\geq J$. For any $s_j\in SH^0(M_j,L^N)$, suppose $s_j=\sum_{k=1}^{d_{j,L^N}}a_kS_{j_k}$ for some $a=(a_1,\dots,a_{d_{j,L^N}})\in S^{2d_{j,L^N}-1}\subset\C^{d_{j,L^N}}$. Then over $q_j^{-1}(U)$, $s_j=(\sum_{k=1}^{d_{j,L^N}}a_kf_{j_k})e^N_{j,L}=\langle a,\bar{f}_j\rangle e^N_{j,L}$. By Poincar\'{e}-Lelong formula, over $q_j^{-1}(U)$, the zero current
\begin{align*}
Z_{s_j}=\frac{\sqrt{-1}}{\pi}\pa\bar{\pa}\log\abs{\langle a,\bar{f}_j\rangle}=\frac{\sqrt{-1}}{\pi}\pa\bar{\pa}\log\abs{\langle a,u_j\rangle}+\frac{\sqrt{-1}}{\pi}\pa\bar{\pa}\log\abs{f_j}=\frac{\sqrt{-1}}{\pi}\pa\bar{\pa}\log\abs{\langle a,u_j\rangle}+\pi^{-1}\Om_{j,L^N},
\end{align*}
where $u_j(z):=\frac{\overline{f_j(z)}}{\abs{f_j(z)}}\in S^{2d_{j,L^N}-1}$.
Then by (\ref{Bergman5}),
\begin{align*}
\underline{Z}_{s_j}-\pi^{-1}\underline{\Om}_{j,L^N}=I_j^{-1}{q_j}_*Z_{s_j}-\pi^{-1}\underline{\Om}_{j,L^N}=I_j^{-1}{q_j}_*\frac{\sqrt{-1}}{\pi}\pa\bar{\pa}\log\abs{\langle a,u_j\rangle}.
\end{align*}
Therefore,
\begin{align*}
\begin{split}
(\underline{Z}_{s_j}-\pi^{-1}\underline{\Om}_{j,L^N},\psi)=&(I_j^{-1}{q_j}_*\frac{\sqrt{-1}}{\pi}\pa\bar{\pa}\log\abs{\langle a,u_j\rangle},\psi) \\
=&(I_j^{-1}\pi^{-1}{q_j}_*\log\abs{\langle a,u_j\rangle},\sqrt{-1}\pa\bar{\pa}\psi) \\
=&\int_{M}(\sqrt{-1}\pa\bar{\pa}\psi(z))(I_j^{-1}\sum_{[\ga]_j\in\Ga/\Ga_j}\pi^{-1}\log\abs{\langle a,u_j([\ga]_jz_j)\rangle}).
\end{split}
\end{align*}
We denote the normalized Haar measure on the sphere $S^{2d_{j,L^N}-1}$ by $\nu_{2d_{j,L^N}-1}$. Following the proof in Theorem 3.1 of \cite{SZ4}, one can show that
\begin{align*}
\begin{split}
&\int_{SH^0(M_j,L^N)}\abs{(\underline{Z}_{s_j}-\pi^{-1}\underline{\Om}_{j,L^N},\psi)}^2\ d\nu_{j,L^N}(s_j) \\
=&\int_{M\times M}(\sqrt{-1}\pa\bar{\pa}\psi(z))(\sqrt{-1}\pa\bar{\pa}\overline{\psi(w)}) \\
&\times I_j^{-2}\sum_{[\ga]_j,[\ga']_j\in\Ga/\Ga_j}\pi^{-2}\int_{S^{2d_{j,L^N}-1}}\log\abs{\langle a,u_j([\ga]_jz_j)\rangle}\log\abs{\langle a,u_j([\ga']_jw_j)\rangle}\ d\nu_{2d_{j,L^N}-1}(a) \\
=&\int_{M\times M}(\sqrt{-1}\pa\bar{\pa}\psi(z))(\sqrt{-1}\pa\bar{\pa}\overline{\psi(w)})I_j^{-2}\sum_{[\ga]_j,[\ga']_j\in\Ga/\Ga_j}\tilde{G}(P_{j,L^N}([\ga]_jz_j,[\ga']_jw_j)) \\
= &\int_{M\times M}(\sqrt{-1}\pa\bar{\pa}\psi(z))(\sqrt{-1}\pa\bar{\pa}\overline{\psi(w)})I_j^{-1}\sum_{[\ga]_j\in\Ga/\Ga_j}\tilde{G}(P_{j,L^N}([\ga]_jz_j,w_j)),
\end{split}
\end{align*}
where
\begin{align*}
P_{j,L^N}(z_j, w_j)
:=&\frac{\abs{\Pi_{j,L^N}(z_j, w_j)}_{h^N}}
{\sqrt{\abs{\Pi_{j,L^N}(z_j, z_j)}_{h^N}}\sqrt{\abs{\Pi_{j,L^N}(w_j, w_j)}_{h^N}}}
\end{align*}
denotes the normalized Bergman kernel of $\Pi_{j,L^N}$ for $j\geq J$ when the denominator never vanishs and the last equality follows from the symmetry. In \cite{SZ4}, Shiffman and Zelditch introduce the function
\begin{align*}
\tilde{G}(t)=-\frac{1}{4\pi^2}\int_0^{t^2}\frac{\log(1-s)}{s}\ ds
\end{align*}
 to calculate the variance and
moreover, they write down the explicit expression of $\tilde{G}$ using power series $$\tilde{G}(t)=\frac{1}{4\pi^2}\sum_{n=1}^\infty\frac{t^{2n}}{n^2},$$ which plays an essential role in our estimate. By the power series expression we have
\begin{align}\label{Variance1}
\tilde{G}(t)\leq\frac{t^2}{24},\qquad\text{ for } 0\leq t\leq1.
\end{align}
Hence by (\ref{Variance1}) and recalling that the denominators of $P_{j,L^N}$'s are bounded from below by a uniform positive constant for $j\geq J$, it follows that
\begin{align}\label{Variance3}
\begin{split}
&\int_{SH^0(M_j,L^N)}\abs{(\underline{Z}_{s_j}-\pi^{-1}\underline{\Om}_{j,L^N},\psi)}^2\ d\nu_{j,L^N}(s_j) \\
\lesssim&\sup_{(z_j,w_j)\in U_j\times U_j}I_j^{-1}\sum_{[\ga]_j\in\Ga/\Ga_j}\abs{\Pi_{j,L^N}([\ga]_jz_j,w_j)}^2_{h^N}\norm{\sqrt{-1}\pa\bar{\pa}\psi}^2_{L^1(M)}.
\end{split}
\end{align}
From now on, for any fixed $(z,w)\in U\times U$ (thus the pairs $(z_j,w_j)\in U_j\times U_j$ and $(\tilde{z},\tilde{w})\in\tilde{U}\times\tilde{U}$ are determined), we always choose the representative $\ga$ of the coset $[\ga]_j\in\Ga/\Ga_j$ such that $dist(\ga\tilde{z},\tilde{w})=\inf_{g\in[\ga]_j}dist(g\tilde{z},\tilde{w})$ (in fact, $\inf_{g\in[\ga]_j}dist(g\tilde{z},\tilde{w})=\min_{g\in[\ga]_j}dist(g\tilde{z},\tilde{w})$ due to the proper discontinuity of deck transformation). With these settings, we can proceed the estimate in (\ref{Variance3}) as follows.

First of all, since $N\geq N^*\geq N_0(M,L,h)$, (\ref{LZThm}) shows that
\begin{align*}
\abs{\Pi_{j,L^N}([\ga]_jz_j,w_j)}_{h^N}\leq\sum_{\ga_j\in\Ga_j}\abs{\tilde{\Pi}_{L^N}(\ga_j\ga\tilde{z},\tilde{w})}_{h^N}.
\end{align*}
If $dist(\ga\tilde{z},\tilde{w})\leq\half\tau_j$, then for any $\ga_j\in\Ga_j\setminus\{1\}$,
\begin{align*}
dist(\ga_j\ga\tilde{z},\tilde{w})\geq dist(\ga_j\ga\tilde{z},\ga\tilde{z})-dist(\ga\tilde{z},\tilde{w})\geq\tau_j-\half\tau_j=\half\tau_j.
\end{align*}
However, if $dist(\ga\tilde{z},\tilde{w})>\half\tau_j$, then $dist(\ga_j\ga\tilde{z},\tilde{w})\geq\half\tau_j$ for all $\ga_j\in\Ga_j$. Similarly as in the proof of Theorem \ref{BK}, we shall have: for all $j\geq J$, which replaces the position of $a_{K,K'}$, there exists $C_{\tilde{U}}>0$(playing the same role as $C_{K,K'}$), such that
\begin{align*}
\abs{\Pi_{j,L^N}([\ga]_jz_j,w_j)}_{h^N}\leq
\begin{cases}
\abs{\tilde{\Pi}_{L^N}(\ga\tilde{z},\tilde{w})}_{h^N}+C_{\tilde{U}}e^{-\si\tau_j}&\text{\quad if }dist(\ga\tilde{z},\tilde{w})\leq\half\tau_j, \\
C_{\tilde{U}}e^{-\si\tau_j}&\text{\quad if }dist(\ga\tilde{z},\tilde{w})>\half\tau_j.
\end{cases}
\end{align*}
As a matter of fact, formulas defining $C_{K,K'}$ and the fact that the replacement for $a_{K,K'}$ is a fixed constant independent of $\tilde{U}$ show that $C_{\ga\tilde{U}}$ can be chosen to be equal to $C_{\tilde{U}}$ for any $\ga\in\Ga$. Thus,
\begin{align}\label{Variance4}
I_j^{-1}\sum_{[\ga]_j\in\Ga/\Ga_j}\abs{\Pi_{j,L^N}([\ga]_jz_j,w_j)}^2_{h^N}
\lesssim C^2_{\tilde{U}}e^{-2\si\tau_j}+I_j^{-1}\sum_{[\ga]_j\in\Ga/\Ga_j,\ dist(\ga\tilde{z},\tilde{w})\leq\half\tau_j}\abs{\tilde{\Pi}_{L^N}(\ga\tilde{z},\tilde{w})}^2_{h^N}.
\end{align}
Let $$A_j(\tilde{z},\tilde{w})=\sum_{[\ga]_j\in\Ga/\Ga_j,\ dist(\ga\tilde{z},\tilde{w})\leq\half\tau_j}\abs{\tilde{\Pi}_{L^N}(\ga\tilde{z},\tilde{w})}^2_{h^N}$$ and denote the coset representatives appearing in the summation of $j$-th step by $\{\ga^{(j)}_1,\dots,\ga^{(j)}_{\kappa_j}\}$, where $1\leq\kappa_j\leq I_j$ since it definitely contains the identity. We observe that $\{\ga^{(j)}_1,\dots,\ga^{(j)}_{\kappa_j}\}\subset\{\ga^{(j+1)}_1,\dots,\ga^{(j+1)}_{\kappa_{j+1}}\}$ because in our convention, representatives of a coset is also a representative of a smaller one and it satisfies the condition for the new summation if it satisfies the previous one. Hence $A_{j+1}(\tilde{z},\tilde{w})$ is obtained from $A_j(\tilde{z},\tilde{w})$ by adding $\De_j=\kappa_{j+1}-\kappa_j$ new terms. We have already shown that those $\ga\in\Ga$ with $d(\ga\tilde{z},\tilde{w})\leq\half\tau_j$ are exactly those representatives appearing in the summation of $A_j$, thus $\half\tau_j<\ga\leq\half\tau_{j+1}$ for $\ga\in\{\ga^{(j+1)}_1,\dots,\ga^{(j+1)}_{\kappa_{j+1}}\}\setminus\{\ga^{(j)}_1,\dots,\ga^{(j)}_{\kappa_j}\}$. Therefore, by Agmon estimates,
\begin{align*}
A_{j+1}(\tilde{z},\tilde{w})\leq\Delta_je^{-\be\sqrt{N}\tau_j}+A_j(\tilde{z},\tilde{w}).
\end{align*}
Denote $\si'=\si'(M,L,h,N)=\be\sqrt{N}>0$. So for all $j\geq J$ and $k\geq1$,
\begin{align*}
\begin{split}
A_{j+k}(\tilde{z},\tilde{w})\leq&\Delta_{j+k-1}e^{-\si'\tau_{j+k-1}}+A_{j+k-1}(\tilde{z},\tilde{w}) \\
\leq&\Delta_{j+k-1}e^{-\si'\tau_{j+k-1}}+\Delta_{j+k-2}e^{-\si'\tau_{j+k-2}}+A_{j+k-2}(\tilde{z},\tilde{w}) \\
\leq&\qquad\qquad\cdots\cdots\cdots \\
\leq&\Delta_{j+k-1}e^{-\si'\tau_{j+k-1}}+\Delta_{j+k-2}e^{-\si'\tau_{j+k-2}}+\cdots+\Delta_je^{-\si'\tau_j}+A_j(\tilde{z},\tilde{w}) \\
\leq&(\Delta_{j+k-1}+\Delta_{j+k-2}+\cdots+\Delta_j)e^{-\si'\tau_j}+A_j(\tilde{z},\tilde{w}) \\
=&(\kappa_{j+k}-\kappa_j)e^{-\si'\tau_j}+A_j(\tilde{z},\tilde{w}) \\
\leq&I_{j+k}e^{-\si'\tau_j}+A_j(\tilde{z},\tilde{w}).
\end{split}
\end{align*}
Hence,
\begin{align*}
0\leq\frac{A_{j+k}(\tilde{z},\tilde{w})}{I_{j+k}}\leq e^{-\si'\tau_j}+\frac{A_j(\tilde{z},\tilde{w})}{I_{j+k}}\leq e^{-\si'\tau_j}+\sup_{\tilde{z}\in\tilde{U}}\abs{\tilde{\Pi}_{L^N}(\tilde{z},\tilde{z})}^2_{h^N}\frac{I_j}{I_{j+k}}\leq e^{-\si'\tau_j}+2^{-k}\sup_{\tilde{z}\in\tilde{U}}\abs{\tilde{\Pi}_{L^N}(\tilde{z},\tilde{z})}^2_{h^N},
\end{align*}
where the last inequality is due to the fact that $\frac{I_{j+k}}{I_j}=[\Ga_{j+2}:\Ga_{j+1}]\cdots[\Ga_{j+k-1}:\Ga_{j+k}]\geq2^k$. Therefore, for any $j\geq0$, we have the uniform estimate
\begin{align}\label{Variance5}
0\leq\frac{A_j(\tilde{z},\tilde{w})}{I_j}=I_j^{-1}\sum_{[\ga]_j\in\Ga/\Ga_j,\ dist(\ga\tilde{z},\tilde{w})\leq\half\tau_j}\abs{\tilde{\Pi}_{L^N}(\ga\tilde{z},\tilde{w})}^2_{h^N}\leq \exp\{-\si'\tau_{\lfloor\frac{j}{2}\rfloor}\}+2^{-\lfloor\frac{j}{2}\rfloor}\sup_{\tilde{z}\in\tilde{U}}\abs{\tilde{\Pi}_{L^N}(\tilde{z},\tilde{z})}^2_{h^N}.
\end{align}
Combining (\ref{Variance3}), (\ref{Variance4}) and (\ref{Variance5}), we get
\begin{align}\label{Variance6}
\begin{split}
&\int_{SH^0(M_j,L^N)}\abs{(\underline{Z}_{s_j}-\pi^{-1}\underline{\Om}_{j,L^N},\psi)}^2\ d\nu_{j,L^N}(s_j) \\
\lesssim&[C^2_{\tilde{U}}\exp\{-2\si\tau_j\}+\exp\{-\si'\tau_{\lfloor\frac{j}{2}\rfloor}\}
+2^{-\lfloor\frac{j}{2}\rfloor}\sup_{\tilde{z}\in\tilde{U}}\abs{\tilde{\Pi}_{L^N}(\tilde{z},\tilde{z})}^2_{h^N}]
\ \norm{\sqrt{-1}\pa\bar{\pa}\psi}_{L^1(M)}^2.
\end{split}
\end{align}
Hence the variance estimate follows. The second statement holds since $\tau_j\to\infty$.
\end{proof}

Any sequence of sections ${\bf{s}}_{L^N}=\{s_j\}_{j=0}^\infty$ with $s_j\in SH^0(M_j,L^N)$ for each $j\geq0$ can be identified as a random element in the probability space $\langle\Pi_{j=0}^\infty SH^0(M_j,L^N),\nu_{L^N}\rangle$, where $\nu_{L^N}$ is the infinite product measure induced by $\nu_{j,L^N}$'s. If we fix an orthonormal basis $\{e_1,\cdots,e_{d_{j,L^N}}\}$ of $H^0(M_j,L^N)$, the set of orthonormal bases $\mathcal{ONB}_{j,L^N}$ of $H^0(M_j,L^N)$ is identical to $U(d_{j,L^N})$, the unitary group of rank $d_{j,L^N}$. Using $\vartheta_{j,L^N}$ to denote the unit mass Haar measure on $\mathcal{ONB}_{j,L^N}$, then $\langle\mathcal{ONB}_{j,L^N},\vartheta_{j,L^N}\rangle$ is a probability space. Similar as above, we may consider a sequence of orthonormal bases ${\bf{S}}_{L^N}=\{(S_{j,1},\dots,S_{j,d_{j,L^N}})\}_{j=0}^\infty\in\langle\Pi_{j=0}^\infty\mathcal{ONB}_{j,L^N},\vartheta_{L^N}\rangle$, where $\vartheta_{L^N}$ is the infinite product measure induced by ${\vartheta_{j,L^N}}$. For all $j\geq0$, denote
$$\lfloor\underline{Z}_{{\bf{s}}_{L^N}}\rfloor_j=\underline{Z}_{s_j}\in\dcal'^{1,1}(M).$$
Then similar as Theorem 1.1 and 1.2 in \cite{SZ1}, we have
\begin{cor}
Assume that $\{\tau_j\}$ defined in (\ref{tau0}) satisfies
\begin{align}\label{Variance7}
\sum_je^{-s\tau_j}<\infty,
\end{align}
for some constant $s>0$. Then there exists $\hat{N}=\hat{N}(M,L,h)>0$ such that for all $N\geq\hat{N}$, 
\begin{enumerate}[i)]
\item $\lfloor\underline{Z}_{{\bf{s}}_{L^N}}\rfloor_j$ converges to $\pi^{-1}\underline{\tilde{\Om}}_{L^N}$ for $\nu_{L^N}$-almost all ${\bf{s}}_{L^N}\in\Pi_{j=0}^\infty SH^0(M_j,L^N)$;
\item For $\vartheta_{L^N}$-almost all ${\bf{S}}_{L^N}=\{(S_{j,1},\dots,S_{j,d_{j,L^N}})\}_{j=0}^{\infty}\in\Pi_{j=0}^\infty\mathcal{ONB}_{j,L^N}$, $$d_{j,L^N}^{-1}\sum_{k=1}^{d_{j,L^N}}\abs{(\underline{Z}_{S_{j,k}}-\underline{\tilde{\Om}}_{L^N},\psi)}^2\to0$$ for any $\psi\in\dcal^{n-1,n-1}(M)$.
Equivalently, for each $j\geq0$ there exists a subset $\La_{j,L^N}\in\{1,\dots,d_{j,L^N}\}$ such that $\frac{\sharp\La_{j,L^N}}{d_{j,L^N}}\to1$ and for any $k\in\La_{j,L^N}$, the sequence $\underline{Z}_{S_{j,k}}$ satisfies
    \begin{align*}
    \lim_{j\to\infty}\underline{Z}_{S_{j,k}}=\pi^{-1}\underline{\tilde{\Om}}_{L^N}.
    \end{align*}
\end{enumerate}
\end{cor}
\begin{proof}
i) Take $\hat{N}(M,L,h)\geq N^*(M,L,h)$ such that the $c(M,L,h,N)$ in Proposition \ref{Variance} satisfies $c(M,L,h,\hat{N})\geq s$. Then for any $N\geq\hat{N}$, choosing any $\psi\in\dcal^{n-1,n-1}(M)$, Proposition \ref{Variance} implies that
\begin{align*}
\begin{split}
&\int_{\Pi_{j=0}^\infty SH^0(M_j,L^N)}\sum_{j=0}^\infty\abs{(\lfloor\underline{Z}_{{\bf{s}}_{L^N}}\rfloor_j-\pi^{-1}\underline{\Om}_{j,L^N},\psi)}^2d\nu_{L^N}({\bf{s}}_{L^N}) \\
=&\sum_{j=0}^\infty\int_{SH^0(M_j,L^N)}\abs{(\underline{Z}_{s_j}-\pi^{-1}\underline{\Om}_{j,L^N},\psi)}^2\ d\nu_{j,L^N}(s_j) \\
\lesssim&\sum_{j=0}^\infty[\exp\{-c\tau_{\lfloor\frac{j}{2}\rfloor}\}+2^{-\frac{j}{2}}]\ \norm{\sqrt{-1}\pa\bar{\pa}\psi}_{L^1(M)}^2 \\
\leq&\sum_{j=0}^\infty[\exp\{-s\tau_{\lfloor\frac{j}{2}\rfloor}\}+2^{-\frac{j}{2}}]\ \norm{\sqrt{-1}\pa\bar{\pa}\psi}_{L^1(M)}^2<\infty
\end{split}
\end{align*}
by (\ref{Variance7}). Therefore, $\lfloor\underline{Z}_{{\bf{s}}_{L^N}}\rfloor_j-\pi^{-1}\underline{\Om}_{j,L^N}\to0$ in the sense of currents for $\nu_{L^N}$-almost all ${\bf{s}}_{L^N}=\{s_j\}\in\Pi_{j=0}^\infty SH^0(M_j,L^N)$. Then i) follows from Proposition \ref{BM}.

ii) follows from the same argument in the proof of Theorem 1.2 in \cite{SZ1}.
\end{proof}




\begin{rmk}
(1) Let $\Gamma_j= H_{j,1} \times H_{j,2} \subset \mathbb{Z}^2$ be a discrete lattice and $M_j = \mathbb{C} / \Gamma_j$ be a real two dimensional flat torus. 
If  $\Ga_0\supsetneq\Ga_1\supsetneq\cdots\supsetneq\Ga_j\supsetneq\cdots$ is a tower of normal subgroups and $H_{j, l} \supsetneq H_{j+1, l}$ for all $j,l$,
then $\tau_{j+1}\geq2\tau_j$. Thus, condition (\ref{Variance7}) holds for all $s>0$. \\

(2) Let $M_j$ be a sequence of compact quotients of $SU(n, 1) / S(U(1)\times U(n)) = \mathbb{B}^n$ corresponding to a tower of congruence subgroups $\Gamma(q_j)$ of $G(Q, \cal{L})$ (See section 2.2 \cite{Ye1} for the detailed definition of these subgroups). Then $$\tau_j \geq 2 \log \left\{c \left[ \vol(M_j)^{\frac{2}{n^2+2n}} \right] \right\} \geq 2 \log c + \frac{4j}{n^2+2n} \log 2 + \frac{4}{n^2+2n}\log \vol(M_0) .$$ (cf. Lemma 2.2.1 \cite{Ye1}) Hence for any $s>0,$ condition (\ref{Variance7}) easily follows. \\
\end{rmk}

\section{Appendix}

We include a proof of the Bergman stability as stated in Theorem \ref{BK} by using the standard H\"ormander-Demailly type $L^2$ estimate, for a slightly more general setup (complete noncompact base manifold with bounded geometry). The proof is well known to the experts, which we record  here for its independent interest.

\begin{prop}\label{bsne}
Assume the \kahler\ manifold $(M,\om_h)$ is complete (not necessarily compact), and satisfies the following geometric finite conditions:
\begin{enumerate}[(a)]
\item The sectional curvature of $(M, \om_h)$ is uniformly bounded;
\item The injectivity radius of $(M, \om_h)$ is uniformly bounded from below by $R>0$.
\end{enumerate}
Then there exists some $N_4=N_4(M,L,h)>0$ such that any tower of normal coverings with line bundles $\{(M_j,L^N)\}$ is Bergman stable whenever $N\geq N_4$.
\end{prop}

\begin{proof}

We essentially follow the argument of \cite{To} (see also \cite{CF}\cite{Ye3}) to break the argument into two parts:
\begin{enumerate}[(i)]
\item $\di\limsup_{j\to\infty}\abs{\Pi_{j,L^N}\left(p_j(z),p_j(z)\right)}_{h^N}\leq\abs{\tilde{\Pi}_{L^N}(z,z)}_{h^N}$ for any $z\in\tilde M$ and any $N\geq1$;
\item $\di\liminf_{j\to\infty}\abs{\Pi_{j,L^N}\left(p_j(z),p_j(z)\right)}_{h^N}\geq\abs{\tilde{\Pi}_{L^N}(z,z)}_{h^N}$ for any $z\in\tilde M$ and any $N\geq N_4$.
    \end{enumerate}

Part (i) follows by a straightforward normal family argument (cf. \cite{To} \cite{CF}\cite{Ye3}) which we will omit here,
while part (ii) is a combination of H\"ormander's $L^2$-estimate and Agmon estimate. 
For any $z\in\tilde{M}$, define $\tau_j(z)=\inf\left\{dist(z,\gamma_j z):\ \gamma_j\in\Gamma_j\setminus\{1\}\right\}$. 
Then $p_j|_{B(z,\half\tau_j(z))}$ is one-to-one and $p_j|_{B(z,\half\tau_j(z))}:B(z,\half\tau_j(z))\to p_j\left(B(z,\half\tau_j(z))\right)$ is a biholomorphism. It is proved in \cite{DW} that $\tau_j(z) \rightarrow \infty$ uniformly on compact subsets of $\tilde M$, as $j \rightarrow \infty$.

\medskip

Now fix a point $x \in\tilde{M}$. We only need to show the case that $\tilde{\Pi}_{L^N}(x,x)\neq0$. Let $\rho(\cdot)=dist(\cdot, x)\in\ccal^0(\tilde{M})$ and $x_j=p_j(x) \in M_j$ for any $j \geq 0$.

\medskip

\textbf{Step 1}: Define sections $\{T_j\in\Gamma(M_j,L^N)\}$.

Consider the coherent state $$S_{x}(y):=\frac{\tilde{\Pi}_{L^N}(y, x)}{\sqrt{\tilde{\Pi}_{L^N}(x, x)}}.$$ Then $S_{x} \in SH^0(\tilde{M},L^N)$ and $\abs{S_{x}(x)}_{h^N}^2=\abs{\tilde{\Pi}_{L^N}(x, x)}_{h^N}$. For any $j\geq0$, let $\tilde{T}_j(y)=\chi_j(\rho(y))S_{x}(y)\in\Gamma(\tilde{M},L^N)$, where the nonincreasing function $\chi_j\in\ccal_c^{\infty}([0,\infty),\R^+)$ satisfies $\chi_j(r)=1$ for $0\leq r\leq\frac{1}{4}\tau_j(x)$, $\chi_j(r)=0$ for $r\geq\frac{1}{3}\tau_j(x)$ and $\norm{\chi'_j}_{\infty}=O(\tau_j(x)^{-1})$. Since $p_j|_{B(x,\half\tau_j(x))}$ is one-to-one, the sections $\{T_j\in\Gamma(M_j,L^N)\}$ are defined as follows:
\begin{align*}
T_j(z)=
\begin{cases}
\tilde{T}_j\left((p_j|_{B(x,\half\tau_j(x))})^{-1}(z)\right) &\text{if} ~z\in p_j(B(x,\half\tau_j(x))), \\
0 & \text{otherwise}.
\end{cases}
\end{align*}

\medskip
\textbf{Step 2}: Construct potential functions $\{\phi_j\}$ following \cite{Ye2}.

The construction is due to \cite{Ye2}.
Since the injectivity radius of the base manifold $M=M_0$ is bounded from below by $R>0$, then the injectivity radius of $M_j$ at $x_j$ is at least $R$ since injectivity radius is nondecreasing along the tower of coverings. Let $\delta\in\ccal^{\infty}_c([0,\infty),\R^+)$ (fixed and independent of $j$) be a nonincreasing cut-off function satisfying $\delta(r)=1$ for $0\leq r\leq \frac{1}{2}R$ and $\delta(r)=0$ for $r\geq R$. In addition, one can pick up $\delta(r)$ so that $-\frac{2+1}{R} \leq \delta'(r) \leq 0$ and $\left|\delta''(r) \right| \leq \frac{4(2+1)}{r^2}$. 
As $\tau_j(x) \rightarrow \infty$ as $j \rightarrow \infty$, by choosing $j$ sufficiently large, we can always assume that $\tau_j(x) > 4R$.
Define a function on $\tilde M$ by 
$$\phi(y)= \log \left(\frac{4 \rho^2(y)}{R^2} \right) \times \delta(\rho(y)).$$
Then the potential function $\phi_j$ on $M_j$ is defined by

\begin{align*}
\phi_j(z)=
\begin{cases}
n \phi \left((p_j|_{B(x,\half\tau_j(x))})^{-1}(z)\right) &\text{if} ~z\in p_j(B(x,\half\tau_j(x))), \\
0 & \text{otherwise}.
\end{cases}
\end{align*}

As the sectional curvature of $M_j$ is uniformly bounded independent of $j$, by the Hessian comparison theorem \cite{GW}, as shown in \cite{Ye2}, one can control the complex Hessian of $\phi_j$ to have $$\frac{\sqrt{-1}}{2} \partial\bar\partial \phi_j \geq -K \om_h,$$ where the positive constant $K=K(M,L,h)$ is independent of $j$ and the base point $x \in \tilde M$.

\medskip

\textbf{Step 3}: Apply H\"ormander's theorem to solve $\bar{\pa}$-equation $\bar{\pa}T'_j=\bar{\pa}T_j$.

 There exists $N_4'=N_4'(M,L,h) 
>0$, such that  
\begin{align}\label{Curvbd}
NRic(h)+\frac{\sqrt{-1}}{2}\pa\bar{\pa}\phi_j+Ric(\om_h) 
\geq \om_h  ~\text{ for}~  N \geq N_4' .
\end{align}
For $N\geq N_4'$ and sufficiently large $j$, we consider the line bundle $(L^N,h^Ne^{-\phi_j})\to (M_j,dV_h)$. By (\ref{Curvbd}), we apply H\"ormander's $L^2$-estimate for the $\bar\partial$-equation (cf. \cite{Dem} Theorem 5.1). 
There exists $T'_j\in L^2(M_j,(L^N,h^Ne^{-\phi_j}))$, such that $\bar{\pa}T'_j=\bar{\pa}T_j$ with
\begin{align}\label{Hor}
\norm{T'_j}^2_{L^2(h^Ne^{-\phi_j})}=\int_{M_j}\abs{T'_j}^2_{h^N}e^{-\phi_j}dV_h
\leq\int_{M_j}\abs{\bar{\pa}T_j}^2_{(h^N,\om_h)}e^{-\phi_j}dV_h=\norm{\bar{\pa}T_j}^2_{L^2(h^Ne^{-\phi_j})}.
\end{align}
Note that $\bar{\pa} T_j$ is supported in $p_j\left(\bar{B}(x,\frac{1}{3}\tau_j(x))\setminus B(x,\frac{1}{4}\tau_j(x))\right)=p_j(\bar{B}(x,\frac{1}{3}\tau_j(x)))\setminus p_j(B(x,\frac{1}{4}\tau_j(x)))$. For any $z\in p_j(B(x,\half\tau_j(x)))$,
\begin{align*}
\begin{split}
\bar{\pa}T_j(z)&=\bar{\pa}\left[\chi_j\left(\rho\circ(p_j|_{B(x,\half\tau_j(x))})^{-1}(z)\right)S_x\left((p_j|_{B(x,\half\tau_j(x))})^{-1}(z)\right)\right] \\
&=\chi'_j\left(\rho\circ(p_j|_{B(x,\half\tau_j(x))})^{-1}(z)\right)\bar{\pa}\rho\left((p_j|_{B(x,\half\tau_j(x))})^{-1}(z)\right)S_{x}\left((p_j|_{B(x,\half\tau_j(x))})^{-1}(z)\right).
\end{split}
\end{align*}
The distance function $\rho$ is 
is differentiable almost everywhere (away from the cut-locus). Moreover, we have $\abs{\bar{\pa}\rho}_{\om_h}^2=\half\abs{d\rho}_{\om_h}^2=\half$ almost everywhere. Hence for almost every $z\in p_j(B(x,\half\tau_j(x)))$,
\begin{align}\label{Hor0}
\abs{\bar{\pa}T_j(z)}^2_{(h^N,\om_h)}=
\half\Abs{\chi'_j\left((p_j|_{B(x,\half\tau_j(x))})^{-1}(z)\right)}^2\
\Abs{S_{x}\left((p_j|_{B(x,\half\tau_j(x))})^{-1}(z)\right)}^2_{h^N}.
\end{align}
From the definition of $\chi_j$,
\begin{align}\label{Hor1}
\left| \chi'_j\left((p_j|_{B(x,\half\tau_j(x))})^{-1}(z)\right) \right|^2\lesssim\tau_j(x)^{-2}.
\end{align}
Applying Agmon estimate on the support of $\bar{\pa}T_j$, when $N\geq N_0$,
\begin{align}\label{Hor2}
\Abs{S_x\left((p_j|_{B(x,\half\tau_j(x))})^{-1}(z)\right)}^2_{h^N}\lesssim e^{-2\be\sqrt{N}\frac{1}{4}\tau_j(x)}=e^{-\half\be\sqrt{N}\tau_j(x)},
\end{align}
provided that $j\geq0$ is large enough to satisfy $\frac{1}{4}\tau_j(x) \geq1$. 
Combining (\ref{Hor0}),(\ref{Hor1}) and (\ref{Hor2}), 
we have that for $j$ large enough, the following holds almost everywhere in $p_j(B(x,\half\tau_j(x)))$:
\begin{align}\label{Hor4}
\abs{\bar{\pa}T_j}^2_{(h^N,\om_h)}\lesssim\tau_j(x)^{-2}e^{-\half\be\sqrt{N}\tau_j(x)}.
\end{align}
As $\bar{\pa}T_j$ is supported in $p_j(\bar{B}(x,\frac{1}{3}\tau_j(x)))\setminus p_j(B(x,\frac{1}{4}\tau_j(x)))$ and $\phi_j$ is supported in $p_j(B(x,R))$, $\phi_j=0$ in the supported of $\bar{\pa}T_j$ for $j$ large enough. Therefore, for such $j$, by (\ref{Hor4}),
\begin{align*}
\begin{split}
\norm{\bar{\pa}T_j}^2_{L^2(h^Ne^{-\phi_j})}=&\int_{M_j}\abs{\bar{\pa}T_j}^2_{(h^N,\om_h)}e^{-\phi_j}dV_h \\
\lesssim&\tau_j(x)^{-2}e^{-\half\be\sqrt{N}\tau_j(x)}\int_{p_j(B(x,\half\tau_j(x)))}\ dV_h \\
=&\tau_j(x)^{-2}e^{-\half\be\sqrt{N}\tau_j(x)}V({B(x,\half\tau_j(x))}).
\end{split}
\end{align*}
Since the Ricci curvature of $\tilde{M}$ has a lower bound, by Bishop volume comparison theorem, $V(B(x,\half\tau_j(x)))\ dV_h$ grows at most exponentially. In other words, there exists $C=C(M,L,h)>0$ such that $V(B(x,\half\tau_j(x)))\leq e^{\frac{C}{2}\tau_j(x)}$. Hence
\begin{align*}
\norm{\bar{\pa}T_j}^2_{L^2(h^Ne^{-\phi_j})}\lesssim\tau_j(x)^{-2}e^{-\half\be\sqrt{N}\tau_j(x)}e^{\frac{C}{2}\tau_j(x)}=\tau_j(x)^{-2}e^{-\half(\be\sqrt{N}-C)\tau_j(x)}.
\end{align*}
Denote $N_4''=N_4''(M,L,h)=\max\{\lfloor\left(\frac{C+2}{\be}\right)^2\rfloor+1,N_0\}$.
Then for $N\geq N_4''$, $\be\sqrt{N}-C>2$, 
\begin{align}\label{Hor5}
\norm{\bar{\pa}T_j}^2_{L^2(h^Ne^{-\phi_j})}\lesssim\tau_j(x)^{-2}e^{-\tau_j(x)}.
\end{align}
Take $N_4=\max\{N_4',N_4''\}$. 
By the $L^2$-estimate (\ref{Hor}), for $N\geq N_4$ and $j$ large enough, then
\begin{align*}
\norm{T'_j}^2_{L^2(h^N e^{-\phi_j})}=\int_{M_j}\abs{T'_j}^2_{h^N}e^{-\phi_j}dV_h\lesssim\tau_j(x)^{-2}e^{-\tau_j(x)} < \infty.
\end{align*}
As $\phi_j\leq\log4$, $e^{-\phi_j}\geq\frac{1}{4}$, then
\begin{align}\label{Hor6}
\norm{T'_j}^2_{L^2(h^N)}\lesssim\norm{T'_j}^2_{L^2(h^Ne^{-\phi_j})}\lesssim\tau_j(x)^{-2}e^{-\tau_j(x)}\rightarrow 0, ~\text{as}~j \rightarrow \infty.
\end{align}
\textbf{Step 4}: Conclusion.

Let $S_j:=T_j-T'_j$. Then $S_j$ satisfies following properties for $N \geq N_4$ and for $j$ sufficiently large:

\begin{enumerate}[(1)]
\item $\bar{\pa}S_j=\bar{\pa}T_j-\bar{\pa}T'_j=0$. This implies $ S_j\in H^0(M_j,L^N)$ and thus $T'_j\in \Ga(M_j,L^N)$.

\item Since $e^{-\phi_j(z)}\sim\left(\rho\circ(p_j|_{B(x,\half\tau_j(x))})^{-1}(z)\right)^{-2n}=dist(z,x_j)^{-2n}$ near $x_j$, $\abs{T'_j}^2_{h^N}e^{-\phi_j}$ is not locally integrable unless we have $T'_j(x_j)=0$. Therefore $S_j(x_j)=T_j(x_j)-T'_j(x_j)=T_j(x_j)$, which implies that $\abs{S_j(x_j)}^2_{h^N}=\abs{T_j(x_j)}^2_{h^N}=\abs{S_{x}(x)}^2_{h^N}=\abs{\tilde{\Pi}_{L^N}(x, x)}_{h^N}>0$.

\item
\begin{align*}
\begin{split}
0<\norm{S_j}_{L^2(h^N)}=&\norm{T_j-T'_j}_{L^2(h^N)}\leq\norm{T_j}_{L^2(h^N)}+\norm{T'_j}_{L^2(h^N)} \\
\leq&\norm{S_{x}}_{L^2(h^N)}+\norm{T'_j}_{L^2(h^N)} \\
=&1+\norm{T'_j}_{L^2(h^N)}.
\end{split}
\end{align*}
\end{enumerate}

Define $F_j=\frac{S_j}{\norm{S_j}_{L^2(h^N)}}\in SH^0(M_j,L^N)$. Therefore, by the extremal property of Bergman kernel,
\begin{align*}
\abs{\Pi_{j,L^N}\left(p_j(x),p_j(x)\right)}_{h^N}=\abs{\Pi_{j,L^N}(x_j, x_j)}_{h^N}\geq\abs{F_j(x_j)}^2_{h^N}
=\frac{\abs{S_j(x_j)}^2_{h^N}}{\norm{S_j}^2_{L^2(h^N)}}\geq\frac{\abs{\tilde{\Pi}_{L^N}(x, x)}_{h^N}}{(1+\norm{T'_j}_{L^2(h^N)})^2}.
\end{align*}
By (\ref{Hor6}), for $N\geq N_4$,
\begin{align*}
\liminf_{j\to\infty}\abs{\Pi_{j,L^N}\left(p_j(x),p_j(x)\right)}_{h^N}\geq\abs{\tilde{\Pi}_{L^N}(x,x)}_{h^N}.
\end{align*}
Hence part (ii) is proved as $x \in \tilde M$ is arbitrary.
\end{proof}


\bigskip

\noindent Yuan Yuan, yyuan05@syr.edu, \\
Department of Mathematics, Syracuse University, Syracuse, NY 13205, USA.

\medskip

\noindent Junyan Zhu, jyzhu@math.jhu.edu, \\
Department of Mathematics, Johns Hopkins University, Baltimore, MD 21218, USA.


\end{document}